\documentclass{amsart}

\usepackage{amsthm}
\usepackage{amsmath}
\usepackage{amssymb}
\usepackage{color}
\usepackage{cite}
\usepackage{tikz}
\usepackage{float}

\newcommand{\lie}{\mathfrak}
\newcommand{\g}{\lie{g}} 
\newcommand{\h}{\lie{h}} 
\renewcommand{\b}{\lie{b}}
\renewcommand{\i}{\lie{i}} 

\newcommand{\Z}{\mathbb{Z}}

\newcommand{\C}{\mathbb{C}} 
\newcommand{\curs}{\mathcal}
\newcommand{\po}{\preccurlyeq}
\newcommand{\propref}[1]{Proposition \ref{#1}}

\newcommand{\corref}[1]{Corollary \ref{#1}}
\newcommand{\thmref}[1]{Theorem \ref{#1}}
\newcommand{\remref}[1]{Remark \ref{#1}}
\DeclareMathOperator{\ad}{\ensuremath{ad}}
\DeclareMathOperator{\wt}{\ensuremath{wt}}
\DeclareMathOperator{\grwt}{\ensuremath{wt^{gr}_{\h_0}}}

\theoremstyle{definition}

\newtheorem*{defn}{Definition}

\theoremstyle{plain}
\newtheorem{thm}{Theorem}
\newtheorem*{prop}{Proposition}
\newtheorem*{cor}{Corollary}
\newtheorem*{lem}{Lemma}

\theoremstyle{remark}
\newtheorem*{rem}{Remark}
\newtheorem*{ex}{Example}

\begin{document}

\title{ad-Nilpotent positively-graded Borel module subalgebras}

\author{Tim Ridenour}
\address{T.~Ridenour: Department of Mathematics, Northwestern
University}
\email{\tt tbr4@math.northwestern.edu}

\author{Adam Sandler}
\address{A.~Sandler: Columbia University}
\email{\tt ars2240@columbia.edu}

\date{\today}

\begin{abstract} In this paper, we study certain ad-nilpotent subalgebras contained in the non-zero graded portion of a simple $\Z_n$-graded Lie algebra. These subalgebras respect the grading on the Lie algebra and are modules for a Borel subalgebra for the grade-zero Lie subalgebra. We show that semisimple elements in such subalgebras lie in the center of the subalgebra, and we provide a classification of these subalgebras whose weight space decompositions have only non-zero weights.

\end{abstract}
\maketitle

\section*{Introduction}

Let $\g$ be a simple Lie algebra with a $\Z_n$-grading. Let $\g_0$ be the (reductive) Lie subalgebra of $\g$ in grade zero, and let $\b_0$ be a Borel subalgebra for $\g_0$. Let $\g_+$ be the direct sum of the non-zero graded components of $\g$. Then, $\g_+$ is a $\b_0$-module. In this paper, we study $\b_0$-submodules of $\g_+$ which are also ad-nilpotent subalgebras of $\g$ which respect the $\Z_n$-grading on $\g$.

These subalgebras are a generalization of $\ad$-nilpotent ideals of a Borel subalgebra of a simple Lie algebra, which have been studied extensively since the publication of \cite{Kos98}. Since then, several papers have been written on the classification and enumeration of ad-nilpotent ideals of Borel subalgebras and various generalizations (see \cite{CP1}, \cite{CP2}, \cite{CP3}, \cite{Sut}, \cite{Pan1}, \cite{Rig}, \cite{CDR}, \cite{BM} to name a few).

In \cite{Pan}, Panyushev further generalized this problem to consider the so-called abelian $\b_0$-stable subalgebras of $\g_1$ for a $\Z_2$-graded Lie algebra $\g = \g_0 \oplus \g_1$. The classification of these subalgebras was further studied in \cite{CP4} and \cite{Dol}, among others. In particular, in \cite{Dol}, a criterion on the set of $\g_0$-weights of $\g_1$ was given to classify the abelian $\b_0$-stable subalgebras using methods similar to those from \cite{CDR}. In this paper, we generalize the methods of \cite{Dol} to classify the $k$-nilpotent positively-graded $\b_0$-module subalgebras for a simple $\Z_n$-graded Lie algebra $\g$.

At this point, we note a few complications that arise in the general case, where $\g$ is $\Z_n$-graded, which were not present in the classification of the abelian $\b_0$-stable subalgebras for $\Z_2$-graded Lie algebras. First, if $\g_+$ is the non-zero graded portion of $\g$, the non-zero weight spaces of $\g_+$ (considered as a $\g_0$-module) may have dimension higher than one. Accordingly, it is possible for a $k$-nilpotent $\b_0$-module subalgebra of $\g_+$ to fail to respect the grading on $\g$; in other words, there exist ad-nilpotent $\b_0$-module subalgebras $\lie{a}$ of $\g_+$ such that $\lie{a} \neq \oplus_{j=1}^{n-1} \lie{a}\cap \g_j$ (see \remref{E2} for example). We avoid this complication by only considering the subalgebras which respect the grading on $\g$.

Second, if $n$ is not prime, an ad-nilpotent positively-graded $\b_0$-module subalgebra $\i$ may contain non-zero elements of weight $0$. The methods of \cite{Dol} do not apply when $0$ is a weight of $\i$. We discuss this possibility at the end of section one and the beginning of section two. In particular, \propref{C1} can be used to show that $\i$ is the direct sum of a subalgebra $\i'$ which has all non-zero weights and a subalgebra of the elements of $\i$ with weight $0$. The subalgebra $\i'$ may then be classified completely by its set of weights in each $\g_j$.

\subsection*{Organization}

In Section 1 of this paper, we provide some preliminary results. We show that any semisimple element in an ad-nilpotent positively-graded $\b_0$-module subalgebra $\i$ is contained in the center of $\i$.

In section two, we provide our main result. We first show that any $k$-nilpotent positively-graded $\b_0$-module subalgebra $\i$ is the direct sum of a central subalgebra and a subalgebra with non-zero $\h_0$-weights. We then provide a complete classification of the $k$-nilpotent positively-graded $\b_0$-module algebras with only non-zero $\h_0$-weights. 

Let $Q_0^+$ be the positive root lattice of (the semisimple portion of) $\g_0$. Let $P_0$ be the set of integral weights of $\g_0$, and let $\curs{P} = P_0 \times \Z_n$. We define a partial order $\po_0$ on $\curs{P}$ as follows: $(\nu,i) \po_0 (\mu,j)$ if $i = j$ and $\mu - \nu \in Q_0^+$. As usual, an antichain of $\curs{P}_0$ is a set of elements which are unrelated in this order.

Assume that $\i$ is a $k$-nilpotent positively-graded $\b_0$-module subalgebra of $\g$. A result from \cite{Pan} shows that a weight space corresponding to a non-zero weight of a graded component of a simple Lie algebra $\g$ is one-dimensional. Any positively-graded $\b_0$-module subalgebra with non-zero weights is therefore a direct sum of one-dimensional weight spaces. Thus, $\i$ determines a set of weights in each graded portion of $\g$, and we call the collection of all of these weights along with the grade for each weight the set of {\it graded $\h_0$-weights} of $\i$, $\grwt (\i)$. Let $\curs{A}_{\i}$ be the antichain of minimal elements of $\grwt (i)$. The set $\grwt (\i)$ consists precisely of the weights in $\grwt (\g_+)$ which are larger than some element in $\curs{A}_{\i}$. We give a necessary and sufficient condition for an antichain of non-zero weights in $\grwt (\g_+)$ to define a $k$-nilpotent positively-graded $\b_0$-module subalgebra.

Finally, in section three of the paper, we provide full classifications of abelian positively-graded $\b_0$-module subalgebras with non-zero weights in small cases. In particular, we provide a full classification for $\g$ of type $A_3$ with a $Z_3$-grading, up to conjugacy. We also provide a $\Z_4$-graded example for $A_3$ with the grading arising from an outer automorphism.

\subsection*{Acknowledgments} 

The first author is grateful to RJ Dolbin for sharing the results of his thesis and for fruitful discussions. We are grateful to Volodymyr Mazorchuk for bringing \cite{BM} to our attention.

\section{Preliminaries}

\noindent We first fix some notation that will be used throughout the paper. We describe the set of complex numbers (resp. integers, non-negative integers, integers modulo $n$) by $\C$ (resp. $\Z$, $\Z_+$, $\Z_n$). We write $i \in \Z_n$ for any element of $\Z_n$ instead of $\overline{i}$. Let $\g$ be a Lie algebra over $\C$ with Cartan subalgebra $\h$. The set of roots (resp., positive roots) for $\g$ are denoted by $R$ (resp., $R^+$). Let $I$ be an indexing set for a set of simple roots for $\g$. Fix a set of Chevalley generators for $\g$: $\{x_{\pm \alpha} \ | \ \alpha \in R^+\} \cup \{h_i \ | \ i \in I\}$.  The Killing form on $\g$ is denoted by $\kappa:\g \otimes \g \to \C$. 

\subsection{}

The following proposition is standard. We supply a proof for the ease of the reader.

\begin{prop} Every $\Z_n$-grading on  a Lie algebra corresponds to a Lie algebra automorphism of $\g$ with order $n$. Furthermore, every automorphism of $\g$ with order $n$ gives a $\Z_n$ grading on $\g$.

\end{prop}

\begin{proof} Let $\g = \displaystyle\bigoplus_{i=0}^{n-1} \g_i$ be a $\Z_n$-grading on $\g$. Let $\zeta$ be a primitive $n^{th}$ root of unity. Define $\sigma:\g \to \g$ by linearly extending the assignment given by \[\sigma(x) = \zeta^ix, \ x\in\g_i. \] For $x \in \g_i$ and $y\in\g_j$, $[x,y]\in\g_{i+j}$, which gives \[\sigma([x,y]) = \zeta^{i+j}[x,y].\]  Also, $$[\sigma(x),\sigma(y)] = [\zeta^i x, \zeta^j y] = \zeta^{i+j}[x,y],$$ so $\sigma$ is a Lie algebra automorphism on $\g$.

\noindent Now, suppose that $\varphi:\g\to\g$ is an automorphism of order $n$. For $0 \leq i \leq n-1$, define \[\g_i := \{x \in \g \ | \ \varphi(x) = \zeta^{i}x\}.\] Then, $\g = \g_0 \oplus \g_1 \oplus \dots \oplus \g_{n-1}$ gives a $\Z_n$ grading on $\g$.

\end{proof}

\subsection{}

\noindent Let $\g$ be of type $X_N$, and let $\sigma$ be a finite-order automorphism of $\g$. We note the following (Theorem 8.6 from {\cite{Kac}}):

\begin{thm}\label{Kac} Let ${\bf s} = (s_0,s_1,...,s_{\ell};r)$ where $s_0,s_1,...,s_{\ell}$ are nonnegative relatively prime integers and $r\in\{1,2,3\}$. Set $n = r\sum_{i = 0}^\ell a_is_i$.

\begin{enumerate} \item[$(i)$] The relations \[\sigma_{\bf s}(E_j) = e^{2\pi i s_j/n}E_j, \ (j = 0,1,...,\ell)\] define (uniquely) an $n^{th}$ order automorphism of $\g$.

\

\item[$(ii)$] Every $n^{th}$ order automorphism of $\g$ is conjugate to some $\sigma_{\bf s}$.

\

\item[(iii)] Let ${\bf s'} = (s_0',s_1',...,s_\ell';r')$. Then, $\sigma_{\bf s}$ is conjugate to $\sigma_{\bf s'}$ by a $\g$-automorphism if and only if $r = r'$ and the sequence $(s_0,...,s_\ell)$ can be transformed to $(s_0',...,s_\ell')$ by an automorphism of the affine Dynkin diagram of type $X_N^{(r)}$.

\end{enumerate}

\end{thm}

\subsection{}

\noindent Let $\g = \g_0 \oplus \g_1 \oplus \dots \oplus \g_{n-1}$ be the $\Z_n$-grading on $\g$. Define \[\g_+ := \bigoplus_{i=1}^{n-1}\g_i.\] By the previous theorem, we may assume that $\h_0 = \h \cap \g_0$ is a Cartan subalgebra of $\g_0$. Note that both $\left.\kappa\right|_{\g_0}: \g_0 \otimes \g_0 \to \C$ and $\left.\kappa\right|_{\h_0}: \h_0 \otimes \h_0 \to \C$ are non-degenerate. We denote the non-degenerate, symmetric invariant bilinear form on $\h_0^*$ induced by the restriction of the Killing form on $\g$ to $\h_0$ by \[( \ , \ )_0: \h_0^* \otimes \h_0^* \to \C\]

Let $\b_0 \subset \g_0$ be a Borel subalgebra of $\g_0$ that contains $\h_0$. We will now define our primary objects of interest.

\begin{defn} A subset $\lie{i} \subset \g_+$ is called a {\it $k$-nilpotent positively-graded $\b_0$-module subalgebra} if $[\lie{i},\lie{i}] \subset \lie{i}$, $[\b_0,\lie{i}] \subset \lie{i}$, $(\ad\lie{i})^k(\lie{i}) = 0$ and $\lie{i} = \bigoplus_{j=1}^{n-1} (\lie{i}\cap\g_j)$.

\end{defn}

\subsection{}

In the remainder of this section, we consider semisimple elements in $\g_+$ and state an important preliminary result concerning the semisimple elements inside of $k$-nilpotent positively-graded $\b_0$-module subalgbras.

\begin{prop}\label{P1} Suppose that $h \in \g_j$ is semisimple and $[\g_0,h] = 0$. Then, $h$ is in the center of the subalgebra $\g^{(j)} := \displaystyle \bigoplus_m \g_{mj}$. In particular, if $j$ is a unit in $\Z_n$, $h=0$.

\end{prop}

\begin{proof} The proof proceeds by induction with the base case given by assumption. Suppose that $[\g_{mj}, h] = 0$. Then, \[\kappa([\g_{-(m+1)j},h],\g_{mj}) = \kappa(\g_{-(m+1)j},[h,\g_{mj}]) = 0.\] By the non-degeneracy of the Killing form, $[\g_{-(m+1)j},h] = 0$. Furthermore, since the centralizer of $h$ is reductive, we have $[\g_{(m+1)j},h] = 0$.

In the case that $j$ is a unit in $\Z_n$, $\g = \oplus_m \g_{mj}$. Since $\g$ is simple, the center is $0$; hence, $h = 0$.

\end{proof}

\subsection{}

\begin{prop} \label{C1} Suppose that $\i$ is a $k$-nilpotent positively-graded $\b_0$-module subalgebra. If $h \in \i\cap\h_j$, then $h$ is in the center of $\g^{(j)}$ and $[h,\i] =0$. Furthermore, if $j$ is a unit in $\Z_n$, then $\i\cap\h_j = 0$.

\end{prop}

\begin{proof} Suppose that $h\in(\i\cap\h_j)$ is not in the center of $\g^{(j)}$. Find $e \in \b_0$ such that $[e,h]\neq 0$, which exists since $[\b_0,h] = 0$ implies $[\g_0,h]=0$, which gives $h = 0$ by Proposition \ref{P1}. 

Since $\i$ is $\b_0$-stable, $[h,e]\in\i$. Furthermore, since $h\in\h$ and $[h,e]\neq 0$, \[0\neq(\ad h)^{k+1}(e) = (\ad h)^k([h,e])\in(\ad \i)^k(\i),\] which contradicts the $k$-nilpotence of $\i$.

A similar argument shows that $[h,\i] = 0$.

\end{proof}

\subsection{}

If $n$ is prime, \propref{C1} shows that any $k$-nilpotent positively-graded $\b_0$-module subalgebra $\i$ contains no non-zero semisimple elements. The following gives an example of an abelian positively-graded $\b_0$-module subalgebra which contains non-zero semisimple elements.

\begin{ex} Suppose that $\g$ is of type $E_6$. We use the enumeration of the nodes of the Dynkin diagram as in \cite{Kac}:

\begin{center}
\begin{figure}[H]\label{Fig1}
  \begin{tikzpicture}[scale=.4]
    \foreach \x in {0,...,4}
    		\draw[thick,xshift=\x cm] (\x cm,0) circle (3 mm);
    \foreach \y in {0,...,3}
    		\draw[thick,xshift=\y cm] (\y cm,0) ++(.3 cm, 0) -- +(14 mm,0);
    \draw[thick] (4 cm,2 cm) circle (3 mm);
    \draw[thick] (4 cm, 3mm) -- +(0, 1.4 cm);
    \draw (4 cm,2.9 cm) node {6}; 
    \draw (0 cm,-8mm) node {1};
    	\draw[xshift= 1 cm] (1 cm,-8mm) node {2};
    	\draw[xshift= 2 cm] (2 cm,-8mm) node {3};
    	\draw[xshift= 3 cm] (3 cm,-8mm) node {4};
    	\draw[xshift= 4 cm] (4 cm,-8mm) node {5};
    	\end{tikzpicture}
  \caption{Dynkin diagram for $E_6$}
\end{figure}
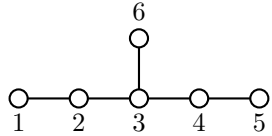
\end{center}

\noindent Consider ${\bf s} = (0,1,0,0,0;2)$. Then, $\sigma_{\bf s}$ is a fourth-order automorphism, which gives rise to a $\Z_4$-grading of $\g$. Let $h = 2h_1 + h_2 - h_4 - 2h_5$. Direct computation shows that $h \in \g_2$ and $[\g_0,h]=0$, so $\i = \C h$ is an abelian positively-graded $\b_0$-module subalgebra.

\end{ex}

\section{Main Theorem}

In this section, we prove our main theorem, classifying all $k$-nilpotent positively-graded $\b_0$-module subalgebras. We first need to introduce some more notation.

Define $R_0$ to be the set of roots (i.e., non-zero $\h_0$-weights) of $\g_0$, and let $\Delta_0$ be a set of simple roots for $\g_0$. For each non-zero $j \in \Z_n$, define $R_j$ to be the set of $\h_0$-weights of $\g_j$. For each $\mu \in R_j$, let $\g_j^{(\mu)}$ be the $\h_0$ weight space of $\g_j$ with $\h_0$-weight $\mu$. Let \[\Sigma_+ = \{(\mu,j) \ | \ \mu \in R_j, \mu \neq 0, \ j\not\equiv 0 \}\] Suppose that $\i$ is a $k$-nilpotent positively-graded $\b_0$-module subalgebra of $\g_+$. By \propref{C1}, we know that $\i = \i'\oplus \lie{c}$ where $\lie{c} \subset \h\cap\g_+$, $[\lie{c},\i] = 0$, and $\i'$ is a $k$-nilpotent positively-graded $\b_0$-module subalgebra with $\i' \cap \g_j^{(0)} = \{0\}$ for all $j \in \Z_n$.

Recall the following lemma from {\cite{Pan}}, which is restated using the notation of the previous paragraph:

\begin{lem} If $(\mu,j)\in\Sigma_+$, then the dimension of the weight space $\g_j^{(\mu)}$ is $1$.

\end{lem}

\subsection{}
Let $Q_0$ (resp., $Q_0^+$) be the $\Z$-span (resp., $\Z_+$-span) of $\Delta_0$, and let $P_0$ be the set of integral weights for $\h_0$. Define $\curs{P}_0 := P_0\times \Z_n$. Then, $\curs{P}_0$ has an addition defined by \[(\eta_1,j_1) + (\eta_2,j_2) = (\eta_1 + \eta_2, j_1 + j_2).\] Define a partial order $\po_0$ on $\curs{P}_0$ via \[(\mu_1,j_1) \po_0 (\mu_2,j_2) \iff j_1 = j_2 \mbox{ and } \mu_2-\mu_1 \in Q_0^+.\]

We will need the following proposition to prove the main theorem:

\begin{prop}\label{P2} Suppose that $\beta_j\in R_{i_j}$, $j = 1, ..., \ell$, and \[\displaystyle \sum_{j=1}^{\ell} (\beta_j, i_j) \po_0 (\lambda,i_1 + ... + i_{\ell}) \mbox{ for some }\lambda \in R_{i_1 + ... + i_{\ell}}.\] Then, there exist $\mu_j \in R_{i_j}$ with $(\beta_j,i_j) \po_0 (\mu_j, i_j)$ for all $j$, $\displaystyle \sum_{j=1}^{\ell} \mu_j \in R_{i_1+...+i_{\ell}}$ and \[(\mu_1 + ... + \mu_{\ell}, i_1 + ... + i_{\ell}) \po_0 (\lambda, i_1 + ... + i_{\ell}).\]

\end{prop}

\begin{proof} The result is trivial if $\beta_1 + ... + \beta_{\ell} \in R_{i_1 + ... + i_{\ell}}$.

\noindent Assume $\beta_1 + ... + \beta_{\ell} \not\in R_{i_1 + ... + i_{\ell}}$. Since \[\sum_{j=1}^{\ell} (\beta_j,i_j) = \left(\beta_1 + ... + \beta_{\ell}, i_1 + ... + i_{\ell}\right) \po_0 (\lambda, i_1 + ... + i_{\ell}),\] we can write \[\lambda - \sum_{j=1}^{\ell} \beta_j = \sum_{k=1}^t \gamma_k\] for some (not necessarily distinct) $\gamma_k \in \Delta_0, k = 1, ..., t$.

\noindent We prove the result by induction on $t$. 

\noindent Suppose $t = 1$. Consider $(\beta_1+ ... + \beta_{\ell}, \gamma_1)_0$. If $(\beta_1 + ... + \beta_{\ell},\gamma_1)_0 < 0$, then $(\beta_r,\gamma_1) < 0$ for some $r$. Define $\mu_r = \beta_r+\gamma_1$  and $\mu_j = \beta_j, j\neq r$. Then, $\mu_j \in R_{i_j}$ for all $j$, $(\beta_j,i_j) \po_0 (\mu_j,i_j)$ for all $j$, and \[\displaystyle \sum_{j=1}^{\ell} \mu_j = \beta_1 + ... + \beta_{\ell} + \gamma_1 = \lambda \in R_{i_1 + ... + i_{\ell}}.\]

\noindent If $(\beta_1 + ... + \beta_{\ell},\gamma_1)_0 \geq 0$, then $(\beta_1 + ...+\beta_{\ell} + \gamma_1,\gamma_1)_0 = (\lambda,\gamma_1)_0 > 0$, which gives \[\lambda - \gamma_1 = \sum_{j=1}^{\ell} \beta_j \in R_{i_1 + ... + i_{\ell}}.\]

\noindent Now, suppose the result is true for all $s < t$. Consider $(\beta_1 + ... + \beta_{\ell},\gamma_1 + ... + \gamma_t)_0$.

\noindent If $(\beta_1 + ... + \beta_{\ell},\gamma_1 + ... + \gamma_t)_0 < 0$, then $(\beta_r, \gamma_m) < 0$ for some $r,m$. Define $\nu_r = \beta_r + \gamma_m$ and $\nu_j = \beta_j, j \neq r$. Then, $\nu_j \in R_{i_j}$ and $(\beta_j,i_j) \po_0 (\nu_j,i_j)$ for all $j$. Furthermore, $\displaystyle \sum_{j=1}^{\ell} (\nu_j,i_j) \po_0 (\lambda, i_1 + ... + i_j)$. Since $\lambda - \displaystyle \sum_{j=1}^{\ell} \nu_j = \gamma_1 + ... + \gamma_{m-1} + \gamma_{m+1} + ... + \gamma_t$, the result follows by the induction hypothesis.

\noindent If $(\beta_1 + ... + \beta_{\ell},\gamma_1 + ... + \gamma_t)_0 \geq 0$, then \[(\beta_1 + ... + \beta_{\ell} + \gamma_1 + ... + \gamma_t, \gamma_1 + ... + \gamma_t)_0 = (\lambda,\gamma_1 + ... + \gamma_t)_0 > 0.\] For some $q$, $(\lambda,\gamma_q)_0 > 0$, which givens $\lambda - \gamma_q \in R_{i_1 + ... + i_{\ell}}$. Set $\lambda' = \lambda - \gamma_q$. This gives $\displaystyle \sum_{j=1}^{\ell} (\beta_j, i_j) \po_0 (\lambda',i_1 + ... + i_{\ell})$, and \[\lambda' - \sum_{j=1}^{\ell} \beta_j = \gamma_1 + ... + \gamma_{q-1} + \gamma_{q+1} + ... + \gamma_t.\] We now may apply the induction hypothesis to obtain the result.

\end{proof}

\subsection{}

Let $\i$ be a $k$-nilpotent positively-graded $\b_0$-module subalgebra of $\g$. For each $j$, $\i_j := \i\cap\g_j$ has an $\h_0$-weight decomposition, and we denote the set of $\h_0$-weights of $\i_j$ by $\wt_{\h_0}(\i_j)$. Define the set of {\it graded $\h_0$-weights} of $\i$, \[\grwt (\i) := \{(\mu,j) \in\curs{P} \ | \ \mu \in \wt_{\h_0} (\i_j)\}.\] If $0$ is not a weight of $\i_j$ for all $j$, then $\grwt (\i)\subset \Sigma_+$. Let $\curs{A}_{\i}$ denote the set of elements in $\grwt (\i)$ which are minimal with respect to $\po_0$. Then, \[\grwt (\i) = \{(\mu,j) \in \Sigma_+ \ | \ (\nu,j) \po_0 (\mu,j) \mbox{ for some } (\nu,j)\in\curs{A}_{\i}\}.\] We say that an antichain $\curs{A}\subset \Sigma_+$ generates a $k$-nilpotent positively-graded $\b_0$-module subalgebra, $\i$, if $\curs{A} = \curs{A}_{\i}$. Moreover, we will abuse terminology and say that $\curs{A}$ is a {\it $k$-nilpotent positively-graded $\b_0$-module subalgebra antichain of $\Sigma_+$} if $\curs{A} \subset \Sigma_+$ generates a $k$-nilpotent positively-graded $\b_0$-module subalgebra.

\begin{thm}\label{T1} An antichain $\curs{A} \subset \Sigma_+$ is a $k$-nilpotent positively-graded $\b_0$-module subalgebra antichain if and only if the following three properties hold: 

\begin{enumerate} \item[(i)] For any $(\mu_1,j_1), (\mu_2,j_2), ..., (\mu_{k+1},j_{k+1}) \in \curs{A}$, $$(\mu_1,j_1) + (\mu_2,j_2) + ... + (\mu_{k+1},j_{k+1}) \not\po_0 (\lambda, j_1 + ... + j_{k+1})$$ for any highest weight $\lambda$ of $\g_{j_1 + ... + j_{k+1}}$.

\item[(ii)] For any $(\mu_1,j_1), (\mu_2,j_2), ..., (\mu_m,j_m) \in \curs{A}$ with $j_1 + ... + j_m \equiv 0$, $$(\mu_1,j_1) + (\mu_2,j_2) + ... + (\mu_m,j_m) \not\po_0 (\theta, 0)$$ for any highest root $\theta$ for $\g_0$.

\item[(iii)] If $(\mu_1, j_1), (\mu_2,j_2) \in \curs{A}$ and $(\mu_1,j_1) + (\mu_2,j_2) \po_0 (\nu,j_1+j_2)$ for some $(\nu, j_1 + j_2) \in \Sigma_+$, there exists $(\alpha,j_1+j_2) \in \curs{A}$ such that $(\alpha, j_1+j_2) \po_0 (\nu, j_1+j_2)$.

 \end{enumerate}

\end{thm}

\begin{proof} Suppose that an antichain $\curs{A}\subset \Sigma_+$ satisfies the three conditions. Let $\lie{i}_{\curs{A}}$ be the vector subspace of $\g_+$ spanned by the $\h_0$-weight vectors in $\g_i$ with weight $\beta$ where $(\alpha,i) \po_0 (\beta,i)$ for some $(\alpha,i) \in \curs{A}$. It is clear that $\lie{i}_{\curs{A}}$ is $\b_0$-stable. For any $k + 1$ weight vectors, $E_{\beta_1}, E_{\beta_2}, ..., E_{\beta_{k+1}}$ in $\lie{i}_{\curs{A}}$, we need to show that 

\[ [E_{\beta_1},[E_{\beta_2},[...,[E_{\beta_k},E_{\beta_{k+1}}]...]]] = 0.\]

\noindent Suppose that $[E_{\beta_1},[E_{\beta_2},[...,[E_{\beta_k},E_{\beta_{k+1}}]...]]] \neq 0$. Then \[(\beta_1 + ... + \beta_{k+1},j_1+...+j_{k+1}) \po_0 (\lambda, j_1+...+j_{k+1})\] for some highest weight $\lambda$ of $\g_{j_1+...+j_{k+1}}$. For each $(\beta_{\ell},j_{\ell})$, $1\leq \ell \leq k+1$, we can choose $(\mu_{\ell},j_{\ell})\in\curs{A}$ such that $(\mu_{\ell},j_{\ell}) \po_0 (\beta_{\ell},j_{\ell})$. Then, $$(\mu_1,j_1) + ... + (\mu_{k+1},j_{k+1}) \po_0 (\beta_1,j_1)+...+(\beta_{k+1},j_{k+1}) \po_0 (\lambda,j_1+...+j_{k+1})$$

\noindent A similar argument shows that $[E_{\eta_1},[E_{\eta_2},[...,[E_{\eta_{m-1}},E_{\eta_m}]...]]] = 0$ for any weight vectors $E_{\eta_{\ell}} \in \g_{j_{\ell}}$ for $\lie{i}_{\curs{A}}$ with $j_1+...+j_m \equiv 0$.

\noindent It remains show that $\lie{i}_{\curs{A}}$ is a Lie subalgebra of $\g_+$. Suppose that $E_{\nu_1}\in\g_{j_1}$ and $E_{\nu_2}\in\g_{j_2}$ are $\h_0$-weight vectors in $\lie{i}_{\curs{A}}$ with $[E_{\nu_1},E_{\nu_2}]\neq 0$. Then, $j_1+j_2 \not\equiv 0$ by the second condition, so $(\nu_1+\nu_2,j_1+j_2) \in \Sigma_+$. Choose $(\mu_{\ell},j_{\ell})\in \curs{A}, \ \ell = 1,2,$ such that $(\mu_{\ell},j_{\ell})\po_0(\nu_{\ell},j_{\ell})$. This gives $(\mu_1,j_1)+(\mu_2,j_2)\po_0 (\nu_1+\nu_2,j_1+j_2)$. Hence, by the third condition, there exists $(\alpha,j_1+j_2) \in \curs{A}$ such that $(\alpha,j_1+j_2)\po_0 (\nu_1+\nu_2,j_1+j_2)$, which shows that $[E_{\nu_1},E_{\nu_2}]\in\lie{i}_{\curs{A}}$ by construction.

\noindent Therefore, $\lie{i}_{\curs{A}}$ is a $k$-nilpotent positively-graded $\b_0$-module subalgebra.

\noindent Now, suppose that $\i \subset \g_+$ is a $k$-nilpotent positively-graded $\b_0$-module subalgebra with $0\not\in \grwt(\i)$. It remains to prove that $\curs{A}_{\i}$ satisfies the given conditions.

\noindent The first two conditions follow immediately from \propref{P2}.

\noindent To show $(iii)$ suppose that $(\mu_1, i_1), (\mu_2,i_2) \in \curs{A}$ and \[(\mu_1,i_1) + (\mu_2,i_2) \po_0 (\nu,i_1+i_2)\] for some $(\nu, i_1 + i_2) \in \Sigma_+$. By \propref{P2}, there exist $\mu'_j \in R_{i_j}$, $j=1,2$, such that $(\mu_j,i_j) \po_0 (\mu'_j,i_j)$, $\mu'_1 + \mu'_2 \in R_{i_1 + i_2}$ and $(\mu'_1 + \mu'_2,i_1 + i_2) \po_0 (\nu,i_1 + i_2)$. Since $\i$ is $\b_0$-stable, it follows that $(\mu'_1 + \mu'_2,i_1 + i_2) \in \grwt(\i)$, and, by definition, there exists $(\alpha,i_1 + i_2) \in \curs{A}_{\i}$ such that $(\alpha, i_1 + i_2) \po_0 (\mu'_1 + \mu'_2, i_1 + i_2) \po_0 (\nu, i_1 + i_2)$.

\end{proof}

\subsection{} 

The statement of the theorem is greatly simplied in the case when $k = 1$ as we note in the following corollary. We use the standard term {\it abelian} in the place of $1$-nilpotent.

\begin{cor}\label{C2} An antichain $\curs{A} \subset \Sigma_+$ is an abelian positively-graded $\b_0$-module subalgebra antichain if and only if for any $(\mu_1,j_1),(\mu_2,j_2) \in \curs{A}$, $$(\mu_1,j_1) + (\mu_2,j_2) \not\po_0 (\lambda,j_1+j_2)$$ for any highest weight $\lambda$ of $\g_{j_1+j_2}$.

\end{cor}

\section{Examples}

In this section, we apply our methods to classify all of the abelian positively-graded $\b_0$-module subalgebras for certain cases where $\g$ is of type $A_3$. We let $R^+ = \{\alpha_1, \alpha_2, \alpha_3, \alpha_1 + \alpha_2, \alpha_2+\alpha_3, \alpha_1+\alpha_2+\alpha_3\}$ be the set of positive roots for $\g$, and let $\{x_{\pm\alpha} \ | \ \alpha \in R^+\} \cup \{h_i \ | \ i = 1,2,3\}$ be a set of Chevalley generators for $\g$.

\subsection{$\Z_3$-gradings on $A_3$}

Let $\g$ be of type $A_3$ with the set of positive roots $R^+$ and Chevalley generators as above. Suppose that $n=3$. By \thmref{Kac}, the automorphisms of $\g$ of order $3$, up to conjugacy, are given by $\sigma_{{\bf s}}$ where ${\bf s}$ is one of the three following possibilities: 

\begin{enumerate} 

\item[] \underline{Case 1:} ${\bf s} = (1,1,1,0;1)$ 

\item[] \underline{Case 2:} ${\bf s} = (2,0,1,0;1)$ 

\item[] \underline{Case 3:} ${\bf s} = (0,0,1,2;1)$ 

\end{enumerate}

\subsubsection{Case 1}

Let ${\bf s} = (1,1,1,0;1)$. Then, $\g_0$ is a reductive with $[\g_0,\g_0]$ of type $A_1$ and a $2$-dimensional center. The following elements form a basis of $\g_j$ for $j\in \Z_3$:

\noindent \underline{$\g_0$}: $x_{\alpha_3}$, $h_1$, $h_2$, $h_3$, $x_{-\alpha_3}$

\

\noindent \underline{$\g_1$}: $x_{\alpha_1}$; $x_{\alpha_2+\alpha_3}, x_{\alpha_2}$; $x_{-\alpha_1-\alpha_2}, x_{-\alpha_1-\alpha_2-\alpha_3}$

\

\noindent \underline{$\g_2$}: $x_{-\alpha_1}$; $x_{-\alpha_2}, x_{-\alpha_2-\alpha_3}$; $x_{\alpha_1+\alpha_2+\alpha_3}, x_{\alpha_1+\alpha_2}$

\

\noindent We choose the Borel subalgebra of $\g_0$ to be the subalgebra $\b_0$ spanned by the set $\{x_{-\alpha_3},h_1,h_2,h_3\}$. The set of simple roots for (the semisimple portion of) $\g_0$ is then $\Delta_0 = \{-\alpha_3\}$. This unique simple root is also the highest root of $\g_0$. The sets $R_1$ and $R_2$ of weights of $\g_1$ and $\g_2$, respectively, are given with the following poset structure with respect to $\po_0$:

\[\begin{array}{cll} R_1: (\alpha_1,1); & (\alpha_2+\alpha_3,1) \po_0 (\alpha_2,1); & (-\alpha_1 -\alpha_2,1) \po_0 (-\alpha_1-\alpha_2-\alpha_3,1) \\ \\ R_2: (-\alpha_1,2); & (-\alpha_2,2) \po_0 (-\alpha_2-\alpha_3,2); & (\alpha_1 + \alpha_2 + \alpha_3,2) \po_0 (\alpha_1 + \alpha_2,2)\end{array}\]

\noindent Using \corref{C2}, we find that there are exactly 30 non-empty abelian positively-graded $\b_0$-module subalgebra antichains in $\Sigma_+$, which we list below:

\begin{enumerate}

\item[1.] $\{( \alpha_1, 1) \}$

\item[2.] $\{( \alpha_2, 1)\}$

\item[3.] $\{(\alpha_2+\alpha_3, 1)\}$

\item[4.] $\{(-\alpha_1-\alpha_2-\alpha_3, 1)\}$

\item[5.] $\{(-\alpha_1-\alpha_2, 1)\}$

\item[6.] $\{(-\alpha_1-\alpha_2-\alpha_3, 1), (\alpha_2, 1)\}$

\item[7.] $\{(-\alpha_1, 2)\}$

\item[8.] $\{(-\alpha_2-\alpha_3, 2)\}$

\item[9.] $\{(-\alpha_2, 2)\}$

\item[10.] $\{(\alpha_1+\alpha_2, 2)\}$

\item[11.] $\{(\alpha_1+\alpha_2+\alpha_3, 2)\}$

\item[12.] $\{(\alpha_1+\alpha_2, 2), (-\alpha_2-\alpha_3, 2)\}$

\item[13.] $\{(\alpha_1, 1), (\alpha_1+\alpha_2, 2)\}$

\item[14.] $\{(\alpha_1, 1), (\alpha_1+\alpha_2+\alpha_3, 2)\}$

\item[15.] $\{(\alpha_1, 1), (-\alpha_2-\alpha_3, 2)\}$

\item[16.] $\{(\alpha_1, 1), (-\alpha_2, 2)\}$

\item[17.] $\{(-\alpha_1, 2), (-\alpha_1-\alpha_2-\alpha_3, 1)\}$

\item[18.] $\{(-\alpha_1, 2), (-\alpha_1-\alpha_2, 1)\}$

\item[19.] $\{(-\alpha_1, 2), (\alpha_2, 1)\}$

\item[20.] $\{(-\alpha_1, 2), (\alpha_2+\alpha_3, 1)\}$

\item[21.] $\{(\alpha_2, 1), (\alpha_1+\alpha_2, 2)\}$

\item[22.] $\{(\alpha_2, 1), (\alpha_1+\alpha_2+\alpha_3, 2)\}$

\item[23.] $\{(\alpha_2+\alpha_3, 1), (\alpha_1+\alpha_2, 2)\}$

\item[24.] $\{(\alpha_2+\alpha_3, 1), (\alpha_1+\alpha_2+\alpha_3, 2)\}$

\item[25.] $\{(-\alpha_2-\alpha_3, 2), (-\alpha_1-\alpha_2-\alpha_3, 1)\}$

\item[26.] $\{(-\alpha_2-\alpha_3, 2), (-\alpha_1-\alpha_2, 1)\}$

\item[27.] $\{(-\alpha_2, 2), (-\alpha_1-\alpha_2-\alpha_3, 1)\}$

\item[28.] $\{(-\alpha_2, 2), (-\alpha_1-\alpha_2, 1)\}$

\item[29.] $\{(-\alpha_1-\alpha_2-\alpha_3, 1), (\alpha_2, 1), (-\alpha_1, 2)\}$

\item[30.] $\{(\alpha_1+\alpha_2, 2), (-\alpha_2-\alpha_3, 2), (\alpha_1, 1)\}$

\end{enumerate}

\subsubsection{Case 2}

Let ${\bf s} = (2,0,1,0;1)$. Then, $\g_0$ is a reductive algebra with a one-dimensional center, and $[\g_0,\g_0]$ is of type $A_1\times A_1$. The following elements form a basis of $\g_j$ for $j\in \Z_3$:

\noindent \underline{$\g_0$}: $x_{\alpha_1}$, $x_{\alpha_3}$, $h_1$, $h_2$, $h_3$, $x_{-\alpha_1}$, $x_{-\alpha_3}$

\

\noindent \underline{$\g_1$}: $x_{\alpha_1+\alpha_2+\alpha_3}$, $x_{\alpha_1+\alpha_2}$, $x_{\alpha_2+\alpha_3}$, $x_{\alpha_2}$

\

\noindent \underline{$\g_2$}: $x_{-\alpha_2}$, $x_{-\alpha_1-\alpha_2}$, $x_{-\alpha_2-\alpha_3}$, $x_{-\alpha_1-\alpha_2-\alpha_3}$

\

\noindent We set the Borel subalgebra of $\g_0$ to be the subalgebra $\b_0$ spanned by the set $\{x_{\alpha_1},x_{\alpha_3},h_1,h_2,h_3\}$. The set of simple roots for (the semisimple portion of) $\g_0$ is then $\Delta_0 = \{\alpha_1,\alpha_3\}$. These two roots are the highest roots of $\g_0$. The sets $R_1$ and $R_2$ of weights of $\g_1$ and $\g_2$, respectively, are given with the following poset structure with respect to $\po_0$, where an arrow $(\mu,i) \to (\nu,i)$ indicates that $(\mu,i)\po_0 (\nu,i)$:

\[\begin{array}{ccc} R_1: & & R_2: \\ \begin{tikzpicture}
  \node (P) {$(\alpha_1 + \alpha_2 + \alpha_3,1) $};
  \node (B) [node distance = 1.25cm, below of = P, left of=P] {$(\alpha_1 + \alpha_2,1)$};
  \node (A) [node distance = 1.25cm, below of=P, right of = P] {$(\alpha_2 + \alpha_3,1)$};
  \node (C) [node distance = 1.25cm, below of=B, right of = B] {$(\alpha_2,1)$};
  \draw[->] (C) to node {} (B);
  \draw[->] (C) to node {} (A);
  \draw[->] (A) to node {} (P);
  \draw[->] (B) to node {} (P);
\end{tikzpicture} & \hspace*{.75cm} & \begin{tikzpicture}
  \node (P) {$(-\alpha_2,2)$};
  \node (B) [node distance = 1.25cm, below of = P, left of=P] {$(-\alpha_1-\alpha_2,2)$};
  \node (A) [node distance = 1.25cm, below of=P, right of = P] {$(-\alpha_2 - \alpha_3,2)$};
  \node (C) [node distance = 1.25cm, below of=B, right of = B] {$(-\alpha_1-\alpha_2-\alpha_3,2)$};
  \draw[->] (C) to node {} (B);
  \draw[->] (C) to node {} (A);
  \draw[->] (A) to node {} (P);
  \draw[->] (B) to node {} (P);
\end{tikzpicture}\end{array}\]

\

\noindent Using \corref{C2}, we find that there are exactly 11 non-empty abelian positively-graded $\b_0$-module subalgebra antichains in $\Sigma_+$, which we list below:

\

\begin{enumerate}

\item[1.] $\{(\alpha_2, 1)\}$

\item[2.] $\{(\alpha_1+\alpha_2, 1)\}$

\item[3.] $\{(\alpha_2+\alpha_3, 1)\}$

\item[4.] $\{(\alpha_1+\alpha_2+\alpha_3, 1)\}$

\item[5.] $\{(-\alpha_1-\alpha_2-\alpha_3, 2)\}$

\item[6.] $\{(-\alpha_1-\alpha_2, 2)\}$

\item[7.] $\{(-\alpha_2-\alpha_3, 2)\}$

\item[8.] $\{(-\alpha_2, 2)\}$

\item[9.] $\{(\alpha_1+\alpha_2, 1), (\alpha_2+\alpha_3, 1)\}$

\item[10.] $\{(-\alpha_1-\alpha_2, 2), (-\alpha_2-\alpha_3, 2)\}$

\item[11.] $\{(-\alpha_2, 2), (\alpha_1+\alpha_2+\alpha_3, 1)\}$

\end{enumerate}

\subsubsection{Case 3}

Consider the $\Z_3$-grading on $\g$ given by the automorphism ${\bf s} = (0,0,1,2;1)$. Then, $\g_0$ has a $1$-dimensional center and $[\g_0,\g_0]$ is of type $A_2$. The following elements form a basis of $\g_j$ for $j\in \Z_3$:

\

\noindent \underline{$\g_0$:} $x_{-\alpha_1-\alpha_2-\alpha_3}$, $x_{\alpha_1}$, $x_{-\alpha_2-\alpha_3}$, $h_1$, $h_2$, $h_3$, $x_{\alpha_1+\alpha_2+\alpha_3}$, $x_{-\alpha_1}$ $x_{\alpha_2+\alpha_3}$

\

\noindent \underline{$\g_1$}: $x_{-\alpha_3}$, $x_{\alpha_1+\alpha_2}$, $x_{\alpha_2}$

\

\noindent \underline{$\g_2$}: $x_{-\alpha_2}$, $x_{-\alpha_1-\alpha_2}$, $x_{\alpha_3}$

\

\noindent We set the Borel subalgebra of $\g_0$ to be the subalgebra $\b_0$ spanned by the set $\{x_{-\alpha_1-\alpha_2 - \alpha_3}, x_{\alpha_1}, x_{-\alpha_2 - \alpha_3},h_1,h_2,h_3\}$. The set of simple roots for (the semisimple portion of) $\g_0$ is then $\Delta_0 = \{-\alpha_1-\alpha_2-\alpha_3, \alpha_1\}$. The highest root of $\g_0$ is then $\theta_0 = -\alpha_2-\alpha_3$. The sets $R_1$ and $R_2$ of weights of $\g_1$ and $\g_2$, respectively, are given with the following poset structure with respect to $\po_0$:

\[\begin{array}{cl} R_1: & (\alpha_2,1) \po_0 (\alpha_1+\alpha_2,1) \po_0 (-\alpha_3,1) \\ \\ R_2: & (\alpha_3,2) \po_0 (-\alpha_1-\alpha_2,2) \po_0 (-\alpha_2,2) \end{array}\]

\

\noindent Using \corref{C2}, we find that there are exactly 6 non-empty abelian positively-graded $\b_0$-module subalgebra antichains in $\Sigma_+$, which we list below:

\

\begin{enumerate}

\item[1.] $\{(-\alpha_3, 1)\}$

\item[2.] $\{(\alpha_1+\alpha_2, 1)\}$

\item[3.] $\{(\alpha_2, 1)\}$

\item[4.] $\{(-\alpha_2, 2)\}$

\item[5.] $\{(-\alpha_1-\alpha_2, 2)\}$

\item[6.] $\{(\alpha_3, 2)\}$

\end{enumerate}

\subsection{An outer automorphism} 

Let $\g$ be of type $A_3$ with the Chevalley generators given at the beginning of the section. Define the following elements of $\g$:

\

\noindent $E_0$=$x_{-\alpha_1-\alpha_2}-x_{-\alpha_2-\alpha_3}$

\noindent $E_1$=$x_{\alpha_1}+x_{\alpha_3}$

\noindent $E_2$=$x_{\alpha_2}$

\noindent $F_0$=$-x_{\alpha_1+\alpha_2}+x_{\alpha_2+\alpha_3}$

\noindent $F_1$=$x_{-\alpha_1}+x_{-\alpha_3}$

\noindent $F_2$=$x_{-\alpha_2}$

\noindent $H_1$=$h_1+h_3$

\noindent $H_2$=$h_2$

\noindent $E_{\gamma_1+\gamma_2}=x_{\alpha_1+\alpha_2}+x_{\alpha_2+\alpha_3}$

\noindent $E_{2\gamma_1+\gamma_2}=x_{\alpha_1+\alpha_2+\alpha_3}$

\noindent $E_{\gamma_0+\gamma_2}=x_{-\alpha_1}-x_{-\alpha_3}$

\noindent $E_{\gamma_0+\gamma_1+\gamma_2}=h_1-h_3$

\noindent $F_{\gamma_1+\gamma_2}=x_{-\alpha_1-\alpha_2}+x_{-\alpha_2-\alpha_3}$

\noindent $F_{2\gamma_1+\gamma_2}=x_{-\alpha_1-\alpha_2-\alpha_3}$

\noindent $F_{\gamma_0+\gamma_2}=-x_{\alpha_1}+x_{\alpha_3}$

\

Consider the $\Z_4$-grading on $\g$ given by the automorphism $\sigma_{(1,1,0;2)}$. Here, $\h_0$ is spanned by $H_1$ and $H_2$. The elements $\gamma_{\ell} \in \h_0^*$, $\ell = 0,1,2$, are given by $[h,E_{\ell}] = \gamma_{\ell}(h)E_{\ell}$ for all $h\in\h_0$. Note that $\gamma_0 + \gamma_1 + \gamma_2 = 0$ as elements in $\h_0^*$. For any $\gamma$, $[h,E_{\gamma}] = \gamma(h)E_{\gamma}$ and $[h,F_{\gamma}] = -\gamma(h)F_{\gamma}$ for all $h\in\h_0$.

Using the above elements, the following form a basis of $\g_j$ for $j\in \Z_4$:

\

\noindent \underline{$\g_0$}: $E_1$, $H_1$, $F_1$; $H_2$

\

\noindent \underline{$\g_1$}: $E_0$; $E_{2\gamma_1+\gamma_2}, E_{\gamma_1+\gamma_2}, E_2$

\

\noindent \underline{$\g_2$}: $F_{\gamma_0+\gamma_2}, E_{\gamma_0+\gamma_1+\gamma_2}, E_{\gamma_0+\gamma_2}$

\

\noindent \underline{$\g_3$}: $F_0$; $F_2, F_{\gamma_1+\gamma_2}, F_{2\gamma_1+\gamma_2}$

\

\noindent We set the Borel subalgebra of $\g_0$ to be the subalgebra $\b_0$ spanned by the set $\{E_1,H_1,H_2\}$. The set of simple roots for (the semisimple portion of) $\g_0$ is then $\Delta_0 = \{\gamma_1\}$. This simple root is also the highest root for $\g_0$. The sets $R_1$, $R_2$ and $R_3$ of weights of $\g_1$, $\g_2$ and $\g_3$, respectively, are given with the following poset structure with respect to $\po_0$:

\[\begin{array}{cl} R_1: & (\gamma_0,1); \\ & (\gamma_2,1) \po_0 (\gamma_1+\gamma_2,1) \po_0 (2\gamma_1 + \gamma_2,1) \\ \\ R_2: & (\gamma_0 + \gamma_2,2) \po_0 (\gamma_0 + \gamma_1 + \gamma_2,2) \po_0 (-\gamma_0 - \gamma_2,2) \\ \\ R_3: & (-\gamma_0,3); \\ & (-2\gamma_1-\gamma_2,3) \po_0 (-\gamma_1-\gamma_2,3) \po_0 (-\gamma_2,3)  \end{array}\]

\

\noindent Using \corref{C2}, we find that there are exactly 20 non-empty abelian positively-graded $\b_0$-module subalgebra antichains in $\Sigma_+$, which we list below:

\

\begin{enumerate}

\item[1.] $\{(\gamma_0, 1)\}$

\item[2.] $\{(2\gamma_1+\gamma_2, 1)\}$

\item[3.] $\{(\gamma_1+\gamma_2, 1)\}$

\item[4.] $\{(\gamma_2, 1)\}$

\item[5.] $\{(-\gamma_0-\gamma_2, 2)\}$

\item[6.] $\{(-\gamma_0, 3)\}$

\item[7.] $\{(-\gamma_2, 3)\}$

\item[8.] $\{(-\gamma_1-\gamma_2, 3)\}$

\item[9.] $\{(-2\gamma_1-\gamma_2, 3)\}$

\item[10.] $\{(2\gamma_1+\gamma_2, 1), (-\gamma_0-\gamma_2, 2)\}$

\item[11.] $\{(\gamma_1+\gamma_2, 1), (-\gamma_0-\gamma_2, 2)\}$

\item[12.] $\{(\gamma_0, 1), (-\gamma_2, 3)\}$

\item[13.] $\{(\gamma_0, 1), (-\gamma_1-\gamma_2, 3)\}$

\item[14.] $\{(\gamma_0, 1), (-2\gamma_1-\gamma_2, 3)\}$

\item[15.] $\{(2\gamma_1+\gamma_2, 1), (-\gamma_2, 3)\}$

\item[16.] $\{(-\gamma_0, 3), (2\gamma_1+\gamma_2, 1)\}$

\item[17.] $\{(-\gamma_0, 3), (\gamma_1+\gamma_2, 1)\}$

\item[18.] $\{(-\gamma_0, 3), (\gamma_2, 1)\}$

\item[19.] $\{(-\gamma_0+\gamma_2, 2), (-\gamma_2, 3)\}$

\item[20.] $\{(-\gamma_0+\gamma_2, 2), (-\gamma_1-\gamma_2, 3)\}$

\end{enumerate}

\

\begin{rem} \label{E2} The elements $E_{\gamma_1+\gamma_2}\in \g_1$ and $F_0\in\g_3$ have the same $\h_0$-weight. Let $v = E_{\gamma_1+\gamma_2} + F_0$. The $\b_0$-submodule of $\g_+$ generated by $v$ is an abelian subalgebra of $\g_+$ that is not positively graded. 

\end{rem}

\bibliography{adambib}{}
\bibliographystyle{plain}

\end{document}